\documentclass[12pt, reqno]{amsart}
\usepackage{amsmath, amsthm, amscd, amsfonts, amssymb, graphicx, xcolor}
\usepackage{graphicx,amssymb,mathrsfs,amsmath,color,fancyhdr,amsthm}
\usepackage[bookmarksnumbered, colorlinks, plainpages]{hyperref}
\usepackage[all,cmtip]{xy}
\usepackage{pgfplots}
\usepackage{tikz}
\usepackage{cite}
\usepackage{circuitikz-1.0}
\usepackage{verbatim}
\usepackage{float}
\usepackage{adjustbox}
\usetikzlibrary{decorations.pathreplacing,decorations.markings}
\usetikzlibrary{arrows, decorations.pathmorphing,backgrounds,positioning,fit,petri}
\usetikzlibrary{decorations.pathreplacing,decorations.markings}
\usetikzlibrary{arrows.meta}
\usepackage{color}

\newtheorem{theorem}{Theorem}[section]
\newtheorem{lemma}[theorem]{Lemma}
\newtheorem{proposition}[theorem]{Proposition}

\theoremstyle{definition}

\newtheorem{example}[theorem]{Example}

\theoremstyle{remark}
\newtheorem{remark}[theorem]{Remark}
\numberwithin{equation}{section}

\usepgfplotslibrary{external}
\begin{document}

\title[Isometry Theorem for Continuous Quiver of Type $\tilde{A}$]{Isometry Theorem for Continuous Quiver of Type $\tilde{A}$}

\author[XIAOWEN GAO$^1$ and MINGHUI ZHAO$^2$$^{*}$]{XIAOWEN GAO$^1$ and MINGHUI ZHAO$^2$$^{*}$}

\address{$^{1}$ School of Science, Beijing Forestry University, Beijing 100083, P. R. China}
\email{gaoxiaowen@bjfu.edu.cn (X.Gao)} 

\address{$^{2}$ School of Science, Beijing Forestry University, Beijing 100083, P. R. China}
\email{zhaomh@bjfu.edu.cn (M.Zhao)}




\thanks{$^{*}$ Corresponding author}
\date{\today}

\begin{abstract}
The Isometry Theorem for continuous quiver of type $A$ plays an important role in persistent homology. In this paper, we shall generalize Isometry Theorem to continuous quiver of type $\tilde{A}$.
\end{abstract} \maketitle

\section{Introduction }
Representation theory of quivers is important in persistent homology and  widely used in topological data analysis.

Gabriel gave the classification of indecomposable representations of quivers of finite type in \cite{Gabriel1972Unzerlegbare}.
In \cite{Crawley2015Decomposition}, Crawley-Boevey gave a
classification of indecomposable representations of $\mathbb{R}$.
In \cite{2017Interval}, Botnan gave a
classification of indecomposable representations of infinite zigzag.
Igusa, Rock and Todorov introduced general continuous quivers of type
$A$ and classified indecomposable representations in \cite{Igusa2022Continuous}.
In \cite{Appel_2020,Appel_2022,Sala_2019,Sala_2021}, Appel, Sala
and Schiffmann introduced continuum quivers independently.
By using these results, persistence diagrams corresponding to representations are defined.

The Isometry Theorem shows that the interleaving distance between two representations of a given quiver is equal to 
the bottleneck distance between corresponding persistence diagrams.
The stability part of Isometry Theorem for continuous quiver $\mathbb{R}$ was given by Cohen-Steiner-Edelsbrunner-Harer in \cite{2007Cohen-Steiner-Edelsbrunner-Harer}, Chazal-De Silva-Glisse-Oudot in \cite{2012The} and Bauer-Lesnick in \cite{2014Bauer-Lesnick}. The converse stability part was given by Bubenik-Scott in \cite{2014Bubenik-Scott} and Lesnick  in \cite{lesnick2015}. In addition, Botnan studied the stability of zigzag persistence modules in \cite{Botnan2018Algebraic}. For the history of Isometry Theorem, see \cite{2015Persistence}.

In \cite{2020Decomposition}, Hanson and Rock introduced continuous quivers of type $\tilde{A}$ and gave a
classification of indecomposable representations. In\cite{Igusa2013,Sala_2019,Sala_2021}, Igusa-Todorov and Sala-Schiffmann also study continuous quivers of type $\tilde{A}$ and its persistence representations.

In \cite{rock2023continuous}, Rock-Zhu built an equivalence between the category of representations of continuous quivers of type $\tilde{A}$ and that of continuous quivers of type $A$ with automorphism.
Based on this equivalence, we shall give an Isometry Theorem for continuous quivers of type $\tilde{A}$ in this paper.

In Section 2, we shall recall Isometry Theorem for continuous quivers of type $A$. The Isometry Theorem for continuous quivers of type $\tilde{A}$ will be given in Section 3 and its proof will be given in Section 4.

\section{Continuous quivers of type ${A}$}

\subsection{Continuous quivers of type ${A}$}

Let $\mathbb{R}$ be the set of real numbers and $<$ be the normal order on $\mathbb{R}$. Following notations in \cite{Igusa2022Continuous}, ${A}_{\mathbb{R}}=(\mathbb{R},<)$ is called a continuous quiver of type $A$.

Let $k$ be a fixed field. A representation of ${A}_{\mathbb{R}}$ over $k$ is given by $\mathbb{V}=(\mathbb{V}(x),\mathbb{V}(x,y))$, where $\mathbb{V}(x)$ is a $k$-vector space for any $x\in\mathbb{R}$ and $\mathbb{V}(x,y):\mathbb{V}(x)\rightarrow \mathbb{V}(y)$ is a $k$-linear map for any $x<y\in\mathbb{R}$ satisfying that $\mathbb{V}(x,y)\circ \mathbb{V}(y,z)=\mathbb{V}(x,z)$ for any $x<y<z\in\mathbb{R}$. 
Let $\mathbb{V}=(\mathbb{V}(x),\mathbb{V}(x,y))$ and  $\mathbb{W}=(\mathbb{W}(x),\mathbb{W}(x,y))$ be two representations of ${A}_{\mathbb{R}}$. A family of $k$-linear maps $\varphi=(\varphi(x))_{x\in \mathbb{R}}$ is called a morphism from $\mathbb{V}$ to $\mathbb{W}$, if $\varphi(y)\mathbb{V}(x,y)$= $\mathbb{W}(x,y) \varphi(x)$, for any $x<y\in \mathbb{R}$.

Denoted by $\mathrm{Rep}_{k}({A}_{\mathbb{R}})$ the category of representations of ${A}_{\mathbb{R}}$ 
and by $\mathrm{Rep}^{pwf}_{k}({A}_{\mathbb{R}})$ the subcategory of pointwise finite-dimensional representations.


For each $a,b\in \mathbb{R}$, we use the notation $|a,b|$ for one of $(a,b),[a,b),(a,b]$ and $[a,b]$. In this paper, we allow $a=-\infty$ or $b=+\infty$, that is the notation $|a,b|$ may mean $(a,+\infty)$, $(-\infty,b)$ or $(-\infty,+\infty)$, too. 

Denote ${T}_{|a,b|}=({T}_{|a,b|}(x),{T}_{|a,b|}(x,y))$ as the following representation of ${A}_{\mathbb{R}}$, where
\begin{equation*}
   {T}_{|a,b|}(x)=\left\{  
\begin{aligned}
k,  & \qquad  x \in |a,b|, \\
0,& \qquad \textrm{otherwise}; \\
\end{aligned}
\right.
\end{equation*}
and
\begin{equation*}
   {T}_{|a,b|}(x,y)=\left\{  
\begin{aligned}
1_k,  & \qquad  x<y \,\, \textrm{and} \,\, x,y \in |a,b|, \\
0,& \qquad \textrm{otherwise}. \\
\end{aligned}
\right.
\end{equation*}
The representation ${T}_{|a,b|}$ is called an interval representation.

Let $\mathbb{V}$ be a pointwise finite-dimensional representation of ${{A}}_{\mathbb{R}}$. In \cite{Crawley2015Decomposition,2018Decomposition,Igusa2022Continuous}, it is proved that it can be decomposed into the direct sum of indecomposable interval representations
\begin{equation*}
  \mathbb{V}=\bigoplus_{i\in I}{T}_{|a_i,b_i|}.
\end{equation*}

\subsection{Isometry Theorem} In this section, we follow the notation in \cite{2015Persistence}.

Let $\mathbb{V}=(\mathbb{V}(x),\mathbb{V}(x,y))$ and $\mathbb{W}=(\mathbb{W}(x),\mathbb{W}(x,y))$ be two representations of ${A}_{\mathbb{R}}$.
For $0\leq\varepsilon\in\mathbb{R}$, a family of linear map $\alpha(x):\mathbb{V}(x) \rightarrow \mathbb{W}(x+\varepsilon)$ is called a morphism of degree $\varepsilon$ from $\mathbb{V}$ to $\mathbb{W}$ if the diagram 
$$\xymatrix{
\mathbb{V}(x)\ar[r]\ar[dr]^{\alpha(x)} & \mathbb{V}(y)\ar[dr]^{\alpha(y)} \\
& \mathbb{W}(x+\varepsilon)\ar[r] & \mathbb{W}(y+\varepsilon)
}$$
is commutative
for any
$x \leq y \in \mathbb{R}$.
Denoted by $\textrm{Hom}_{{A}_{\mathbb{R}}}^{\varepsilon} ({\mathbb{V}},{\mathbb{W}})$
the set of morphisms of  degree  $\varepsilon$ from ${\mathbb{V}}$ to ${\mathbb{W}}$.

An $\varepsilon$-interleaving between $\mathbb{V}$ and $\mathbb{W}$ is two families of morphisms $\alpha\in\textrm{Hom}_{{A}_{\mathbb{R}}}^{\varepsilon}({\mathbb{V}},{\mathbb{W}})$ and $\beta\in\textrm{Hom}_{{A}_{\mathbb{R}}}^{\varepsilon}({\mathbb{W}},{\mathbb{V}})$ such that the diagrams
$$\xymatrix{
\mathbb{V}(x) \ar[rr]\ar[dr]^{\alpha(x)} && \mathbb{V}(x+2\varepsilon) & &\mathbb{V}(x+\varepsilon)\ar[dr]^{\alpha(x+\varepsilon)}&\\
&\mathbb{W}(x+\varepsilon)\ar[ur]^{\beta(x+\varepsilon)}& &\mathbb{W}(x) \ar[rr]\ar[ur]^{\beta(x)} && \mathbb{W}(x+2\varepsilon) &
}$$ are commutative for any
$x\in \mathbb{R}$.

The interleaving distance of representations of ${A}_{\mathbb{R}}$  is defined as
\begin{equation*}
    d_{i,{A}_{\mathbb{R}}}(\mathbb{V},\mathbb{W})=\inf \{ \varepsilon \geq 0 \mid \textrm{there is an $\varepsilon$-interleaving between $\mathbb{V}$ and $\mathbb{W}$}\}.
\end{equation*}

For a pointwise finite-dimensional representation $\mathbb{V}=\bigoplus_{i\in I}{T}_{|a_i,b_i|}$ of ${{A}}_{\mathbb{R}}$. The set of intervals $|a_i,b_i|$ for all $i\in I$ is called the persistence barcodes of $\mathbb{V}$.
The multiset $dgm(\mathbb{V})$ of points with coordinate $(a_i,b_i)$ is called the persistence diagram of $\mathbb{V}$.

Let $A$ and $B$ be two multisets of points in the extended plane $\bar{\mathbb{R}}^2$. A partial matching between $A$ and $B$ is a subset $P$ of $A\times B$ such that
\begin{enumerate}
\item  there is at most one point $b\in B$ such that $(a,b)\in P$ for any $a\in A$;
\item  there is at most one point $a\in A$ such that $(a,b)\in P$ for any $b\in B$.
\end{enumerate}
The bottleneck cost $c(P)$ of the partial matching $P$ is defined as
\begin{equation*}
    c(P)=\max \{ \sup_{(a,b)\in P}||a-b||_\infty, \sup_{s=(s_x,s_y)\in S} {\frac{s_x-s_y}{2}}\},
\end{equation*}
where $S$ is the set of unmatched points in $A\cup B$.
The bottleneck distance  between $A$ and $B$ is defined as
\begin{equation*}
    d_{b,\mathbb{R}^2}(A,B)=\inf_{\textrm{partial matching }P} c(P).
\end{equation*}

\begin{theorem}[\cite{2015Persistence}]\label{thm_iso_A}
Let $\mathbb{V}$ and $\mathbb{W}$ be two pointwise finite-dimensional representations of continuous quiver ${{A}}_{\mathbb{R}}$ of type ${A}$. Then, 
\begin{equation*}
d_{b,\mathbb{R}^2}(dgm(\mathbb{V}),dgm(\mathbb{W}))=d_{i,{A}_{\mathbb{R}}}( \mathbb{V},\mathbb{W}).
\end{equation*}
\end{theorem}

\section{Continuous quivers of type $\tilde{A}$}
\subsection{Continuous quivers of type $\tilde{A}$}

We define an equivalence "$\sim$" on $\mathbb{R}$, where $x\sim y$ if and only if $y-x\in \mathbb{Z}$ for any $x,y\in \mathbb{R}$. Let $[x]$ be the equivalent class of $x$ for any $x\in\mathbb{R}$ and define $[x]<[y]$ if and only if $x<y$ and $|x-y|<\frac{1}{2}$. Let $\tilde{{A}}_{\mathbb{R}}=(\mathbb{R}/{\sim},<)$, which is called a continuous quiver of type $\tilde{A}$. 

A representation of $\tilde{{A}}_{\mathbb{R}}$ over $k$ is given by $\textbf{V}=(\textbf{V}([x]),\textbf{V}([x],[y]))$, where $\textbf{V}([x])$ is a $k$-vector space, and $\textbf{V}([x],[y]):\textbf{V}([x])\rightarrow \textbf{V}([y])$ is a $k$-linear map for any $[x]<[y]\in \mathbb{R}/{\sim}$ satisfying that $\textbf{V}([x],[y]) \circ \textbf{V}([y],[z])=\textbf{V}([x],[z])$ for any $[x]<[y]<[z]\in \mathbb{R}/{\sim}$. Let $\textbf{V}=(\textbf{V}([x]),\textbf{V}([x],[y]))$ and  $\textbf{W}=(\textbf{W}([x]),\textbf{W}([x],[y]))$ be representations of $\tilde{{A}}_{\mathbb{R}}$. A family of $k$-linear maps $f([x]):\textbf{V}([x])\rightarrow \textbf{W}([x])$ is called a morphism from $\textbf{V}$ to $\textbf{W}$, if it satisfies $f([y]) \textbf{V}([x],[y])$= $\textbf{W}([x],[y]) f([x])$, for any $[x]<[y]$.

Let $\textbf{V}=(\textbf{V}([x]),\textbf{V}([x],[y]))$ be a representation of $\tilde{{A}}_{\mathbb{R}}$ over $k$. The representation $\textbf{V}$ are called nilpotent, provided that the linear map $\textbf{V}([x_n],[x])\circ\textbf{V}([x_{n-1}],[x_n])\circ\cdots\circ\textbf{V}([x_1],[x_2])\circ\textbf{V}([x],[x_1])$ is nilpotent for any $[x]<[x_1]<[x_2]<\cdots<[x_{n-1}]<[x_n]<[x]$.

 Denoted by $\mathrm{Rep}^{nil}_{k}(\tilde{{A}}_{\mathbb{R}})$ the category of nilpotent representations of $\tilde{{A}}_{\mathbb{R}}$ and by $\mathrm{Rep}^{pwf}_{k}(\tilde{{A}}_{\mathbb{R}})$ the subcategory of pointwise finite-dimensional representations.

For each $a<b\in\mathbb{R}$, consider a representation $\mathbf{T}_{|a,b|}=(\mathbf{T}_{|a,b|}([x]),\mathbf{T}_{|a,b|}([x],[y]))$ of $\tilde{{A}}_{\mathbb{R}}$, defined by
\begin{equation*}
\textbf{T}_{|a,b|}([x])=\bigoplus_{z\in [x]}{T}_{|a,b|}(z),
\end{equation*}
\begin{equation*}
\textbf{T}_{|a,b|}([x],[y])=\bigoplus_{z_1\in [x],z_2\in[y],|z_1-z_2|<1}{T}_{|a,b|}(z_1,z_2).
\end{equation*}
Since the representation $T_{|a,b|}$ is indecomposable,  $\mathbf{T}_{|a,b|}$ is an indecomposable representation of $\tilde{{A}}_{\mathbb{R}}$. The representation $\mathbf{T}_{|a,b|}$ are called interval representations.

Hanson and Rock proved the following theorem in \cite{2020Decomposition}. 
\begin{theorem}[\cite{2020Decomposition}]\label{thm_HR}
Let $\mathbf{V}$ be a pointwise finite-dimensional nilpotent representation of $\tilde{{A}}_{\mathbb{R}}$. Then, it can be decomposed into the direct sum of indecomposable interval representations
\begin{equation*}
  \mathbf{V}=\bigoplus_{i\in I}\mathbf{T}_{|a_i,b_i|}.
\end{equation*}
\end{theorem}

\subsection{Isometry Theorem}
Let $\mathbf{V}=(\mathbf{V}([x]),\mathbf{V}([x],[y]))$, $\mathbf{W}=(\mathbf{W}([x]),\mathbf{W}([x],[y]))$ be two representations of $\tilde{{A}}_{\mathbb{R}}$.
For $0\leq\varepsilon<1\in\mathbb{R}$,
a family of linear maps $\alpha([x]):\mathbf{V}([x]) \rightarrow \mathbf{W}([x+\varepsilon])$ is called a morphism of degree $\varepsilon$ from $\mathbf{V}$ to $\mathbf{W}$ if the diagram
$$\xymatrix{
\mathbf{V}([x])\ar[r]\ar[dr]^{\alpha([x])} & \mathbf{V}([y])\ar[dr]^{\alpha([y])} \\
& \mathbf{W}([x+\varepsilon])\ar[r] & \mathbf{W}([y+\varepsilon])
}$$ is commutative
for any
$x \leq y \in \mathbb{R}$ such that $|x-y|<\frac{1}{2}$.
Denoted by $\textrm{Hom}_{\tilde{{A}}_{\mathbb{R}}}^{\varepsilon} ({\mathbb{V}},{\mathbb{W}})$
the set of morphisms of  degree  $\varepsilon$ from ${\mathbf{V}}$ to ${\mathbf{W}}$.

An $\varepsilon$-interleaving between $\mathbf{V}$ and $\mathbf{W}$ is two  morphisms $\alpha\in\textrm{Hom}_{\tilde{{A}}_{\mathbb{R}}}^{\varepsilon}({\mathbf{V}},{\mathbf{W}})$ and $\beta\in\textrm{Hom}_{\tilde{{A}}_{\mathbb{R}}}^{\varepsilon}({\mathbf{W}},{\mathbf{V}})$ such that the diagrams 
$$\xymatrix{
\mathbf{V}([x]) \ar[rr]\ar[dr]^{\alpha([x])} && \mathbf{V}([x+2\varepsilon]) & &\mathbf{V}([x+\varepsilon])\ar[dr]^{\alpha([x+\varepsilon])}&\\
&\mathbf{W}([x+\varepsilon])\ar[ur]^{\beta([x+\varepsilon])}& &\mathbf{W}([x]) \ar[rr]\ar[ur]^{\beta([x])} && \mathbf{W}([x+2\varepsilon]) &
}$$
are commutative for any
$x\in \mathbb{R}$.

The interleaving distance of $\tilde{{A}}_{\mathbb{R}}$ representations between $\mathbf{V}$ and $\mathbf{W}$ is defined as
\begin{equation*}
    d_{i,\tilde{{A}}_{\mathbb{R}}}(\mathbf{V},\mathbf{W})=\inf \{ \varepsilon \geq 0 \mid \textrm{there is an $\varepsilon$-interleaving between $\mathbf{V}$ and $\mathbf{W}$}\}.
\end{equation*}

Consider an equivalence "$\sim$" on $\mathbb{R}^2$, where $(x_1,x_2)\sim (y_1,y_2)$ if and only if $y_1-x_1=y_2-x_2\in \mathbb{Z}$ for any $(x_1,x_2), (y_1,y_2)\in \mathbb{R}^2$.
Let $\mathbf{V}=\bigoplus_{i\in I}\mathbf{T}_{|a_i,b_i|}$ be a pointwise finite-dimensional nilpotent representation  of $\tilde{{A}}_{\mathbb{R}}$.
The multiset $dgm(\mathbf{V})$ in $\mathbb{R}^2/{\sim}$ consisting of equivalent classes of points with coordinate $(a_i,b_i)$ is called the persistence diagram of $\mathbf{V}$.

Let $A$ and $B$ be two multisets in $\mathbb{R}^2/{\sim}$. A partial matching between $A$ and $B$ is a subset $P$ of $A\times B$ such that
\begin{enumerate}
\item  there is at most one $\bar{b}\in B$ such that $(\bar{a},\bar{b})\in P$ for any $\bar{a}\in A$;
\item  there is at most one $\bar{a}\in A$ such that $(\bar{a},\bar{b})\in P$ for any $\bar{b}\in B$.
\end{enumerate}
The bottleneck cost $c(P)$ of the partial matching $P$ is defined as
\begin{equation*}
    c(P)=\max \{ \sup_{(\bar{a},\bar{b})\in P}||\bar{a}-\bar{b}||_\infty, \sup_{\bar{s}=\overline{(s_x,s_y)}\in S} {\frac{s_x-s_y}{2}}\},
\end{equation*}
where $||\bar{a}-\bar{b}||_\infty=\min_{a\in\bar{a},b\in\bar{b}}||a-b||_\infty$ and $S$ is the set of unmatched elements in $A\cup B$. 
The bottleneck distance is defined as
\begin{equation*}
    d_{b,\mathbb{R}^2/{\sim}}(A,B)=\inf_{\textrm{partial matching }P} c(P).
\end{equation*}

The following theorem is the main result in this paper.
 
 \begin{theorem}\label{main-thm}
Let  $\mathbf{V}$ and $\mathbf{W}$ be two pointwise finite-dimensional nilpotent representations of $\tilde{{A}}_{\mathbb{R}}$. Then,
\begin{equation*}
d_{b,,\mathbb{R}^2/{\sim}}(dgm(\mathbf{V}),dgm(\mathbf{W}))=d_{i,\tilde{{A}}_{\mathbb{R}}}( \mathbf{V},\mathbf{W})
\end{equation*}
\end{theorem}

The proof of this theorem will be given in the next section.

\section{The proof of the main result}
\subsection{Continuous quivers of type $A$ with automorphism}
In this section, we follow the notation in \cite{2024Gao-Zhao}.

Consider a morphism $\sigma:\mathbb{R}\mapsto \mathbb{R}$ sending $x$ to $x+1$. In this paper, $\textbf{Q}=(A_{\mathbb{R}},\sigma)$ is called a continuous quiver of type $A$ with automorphism $\sigma$. 
A representation of $\textbf{Q}$ over $k$ is given by a representation $\mathbb{V}=(\mathbb{V}(x),\mathbb{V}(x,y))$ of $A_{\mathbb{R}}$ such that
$\mathbb{V}(x)=\mathbb{V}(\sigma(x))$ for any $x\in\mathbb{R}$ and 
$\mathbb{V}(x,y)=\mathbb{V}(\sigma(x),\sigma(y))$ for any $x\leq y \in \mathbb{R}$.

For a representation $\mathbb{V}=(\mathbb{V}(x),\mathbb{V}(x,y))$ of $A_{\mathbb{R}}$, consider a representation $\sigma^{*}(\mathbb{V})=(\sigma^{*}(\mathbb{V})(x),\sigma^{*}(\mathbb{V})(x,y))$, where $\sigma^{*}(\mathbb{V})(x)=\mathbb{V}(\sigma(x))$ and $\sigma^{*}(\mathbb{V})(x,y)=\mathbb{V}(\sigma(x,y))$. Note that $\bigoplus_{k\in\mathbb{Z}}(\sigma^{*})^{k}(\mathbb{V})$ is representation of $\mathbf{Q}$. Let $$\mathbb{T}_{|a,b|}=\bigoplus_{k\in \mathbb{Z}}(\sigma^{*})^{k}(T_{|a,b|}).$$ Since the representation $T_{|a,b|}$ is indecomposable, so is $\mathbb{T}_{|a,b|}$.
The representation $\mathbb{T}_{|a,b|}$ is called an interval representation of $\textbf{Q}$.

Denoted by $\mathrm{Rep}_{k}(\textbf{Q})$ the category of representations of $\textbf{Q}$. Denote $\mathrm{Rep}_{k}^{pwf}(\textbf{Q})$ the subcategory of $\mathrm{Rep}_{k}(\textbf{Q})$ consisting of pointwise finite-dimensional representations.

Note that there exists an equivalence  $$\Psi:\mathrm{Rep}_{k}(\tilde{{A}}_{\mathbb{R}})\rightarrow \mathrm{Rep}_{k}(\mathbf{Q})$$
such that
  $\Psi({\mathbf{V}})(x)={\mathbf{V}}([x])$ for any $x\in\mathbb{R}$ and 
   $\Psi({\mathbf{V}})(x,y)={\mathbf{V}}([x],[y])$
   for any $x<y\in\mathbb{R}$ such that $|x-y|<\frac{1}{2}$.
By definitions, we have $$\Psi(\mathbf{T}_{|a,b|})=\mathbb{T}_{|a,b|}.$$

\begin{remark}
The equivalence $\Psi$ is built by Rock-Zhu in \cite{rock2023continuous}, where the continuous quiver $\textbf{Q}$ of type $A$ with automorphism is called a continuous quiver of type A with $\mathbb{Z}$-action.
\end{remark}

Let $\mathbb{V}$ be a pointwise finite-dimensional representation of $\textbf{Q}$. Theorem \ref{thm_HR} implies that $\mathbb{V}$ can be decomposed into the direct sum of indecomposable interval representations
\begin{equation*}
  \mathbb{V}=\bigoplus_{i\in I}\mathbb{T}_{|a_i,b_i|}.
\end{equation*}

For two representations $\mathbb{V}=(\mathbb{V}(x),\mathbb{V}(x,y))$, $\mathbb{W}=(\mathbb{W}(x),\mathbb{W}(x,y))$ of $\mathbf{Q}$ and $0\leq\varepsilon\in\mathbb{R}$,
let $\textrm{Hom}_{\mathbf{Q}}^{\varepsilon}({\mathbb{V}},{\mathbb{W}})$ be the set consisting of morphisms $\alpha=(\alpha(x):\mathbb{V}(x) \rightarrow \mathbb{W}(x+\varepsilon))$ of degree  $\varepsilon$ such that $\alpha(x)=\alpha(x+1)$.

An $\varepsilon$-interleaving of $\mathbf{Q}$ representations between $\mathbb{V}$ and $\mathbb{W}$ is two families of morphisms $\alpha\in\textrm{Hom}_{\mathbf{Q}}^{\varepsilon}({\mathbb{V}},{\mathbb{W}})$ and $\beta\in\textrm{Hom}_{\mathbf{Q}}^{\varepsilon}({\mathbb{W}},{\mathbb{V}})$ such that  the following diagrams are commutative for any
$x\in \mathbb{R}$.
$$\xymatrix{
\mathbb{V}(x) \ar[rr]\ar[dr]^{\alpha(x)} && \mathbb{V}(x+2\varepsilon) & &\mathbb{V}(x+\varepsilon)\ar[dr]^{\alpha(x+\varepsilon)}&\\
&\mathbb{W}(x+\varepsilon)\ar[ur]^{\beta(x+\varepsilon)}& &\mathbb{W}(x) \ar[rr]\ar[ur]^{\beta(x)} && \mathbb{W}(x+2\varepsilon) &
}$$

The interleaving distance of $\mathbb{V}$ and $\mathbb{W}$ is defined as
\begin{equation*}
    d_{i,\mathbf{Q}}(\mathbb{V},\mathbb{W})=\inf \{ \varepsilon \geq 0 \mid \textrm{there is an $\varepsilon$-interleaving between $\mathbb{V}$ and $\mathbb{W}$}\}.
\end{equation*}

\subsection{The proof of the main result}
Let $\sigma:\mathbb{R}^2\rightarrow\mathbb{R}^2$ be a map sending $(x,y)$ to $(x+1,y+1)$.

\begin{lemma}\label{lemma-inv}
Let $A$ and $B$ be two  $\sigma$-invariant multisets of 
points in $\mathbb{R}^2$ and denote $\bar{A}=\{\bar{a}|a\in A\}$ and $\bar{B}=\{\bar{b}|b\in B\}$. Assume that $\bar{A}$ and $\bar{B}$ are finite multisets.
Then we have
\begin{equation*}
    d_{b,\mathbb{R}^2}(A,B)=d_{b,\mathbb{R}^2/{\sim}}(\bar{A},\bar{B}),
\end{equation*}
\end{lemma}
\begin{proof}
At first, we shall prove 
\begin{equation}\label{eq_3.3}
    d_{b,\mathbb{R}^2}(A,B)\geq d_{b,\mathbb{R}^2/{\sim}}(\bar{A},\bar{B}).
\end{equation}

Assume that $d_{b,\mathbb{R}^2/{\sim}}(\bar{A},\bar{B})=\epsilon$.
Let $P$ be a partial matching between $A$ and $B$. If $c(P)=\infty$, then we have $c(P)\geq\epsilon$.  If $c(P)$ is finite, we shall prove $c(P)\geq\epsilon$, too.

Denoted by $\bar{A}'$ the subset of $\bar{A}$ consisting of $\bar{a}\in\bar{A}$ such that all points in $\bar{a}$ are matched points in $A$.
Let $\bar{B}''$ be the subset of $\bar{B}$ satisfying that
\begin{enumerate}
\item  for any  $\bar{b}\in\bar{B}''$, there exist  some $a'\in\bar{a}\in\bar{A}'$ and some $b'\in\bar{b}$ such that $(a',b')\in P$;
\item  for any $a'\in\bar{a}\in\bar{A}'$ and $b'\in\bar{b}\not\in\bar{B}''$, it holds that $(a',b')\not\in P$.
\end{enumerate}
Since $c(P)$ is finite, we have $|\bar{B}''|\geq|\bar{A}'|$.
Then we can get a bijection $\tau$ between $\bar{A}'$ and a subset $\bar{B}'''$ of $\bar{B}''$ satisfying that
there are some $a'\in\bar{a}$ and some $b'\in\tau(\bar{a})$ with $(a',b')\in P$ for any $\bar{a}\in\bar{A}'$.

Similarly, denoted by $\bar{B}'$ the subset of $\bar{B}$ consisting of $\bar{b}\in\bar{B}$ such that all points in $\bar{b}$ are matched points in $B$.
Let $\bar{A}''$ be a subset of $\bar{A}$ satisfying that
\begin{enumerate}
\item  for any $\bar{a}\in\bar{A}''$, there exist some $a'\in\bar{a}$ and some $b'\in\bar{b}\in\bar{B}'$ such that $(a',b')\in P$;
\item   for any $a'\in\bar{a}\not\in\bar{A}''$ and $b'\in\bar{b}\in\bar{B}'$, it holds that $(a',b')\not\in P$.
\end{enumerate}
Since $c(P)$ is finite, we have $|\bar{A}''|\geq|\bar{B}'|$.
Then we can get a bijection $\theta$ between $\bar{B}'$ and a subset $\bar{A}'''$ of $\bar{A}''$ satisfying that
there are some $a'\in\theta(\bar{b})$ and some $b'\in\bar{b}$ with $(a',b')\in P$ for any $\bar{b}\in\bar{B}'$.

Let $\bar{A}'=\{A_1,A_2,\ldots,A_s\}$. For convenience, assume that $\tau(A_i)\not\in\bar{B}'$ if $i\leq l$ and $\tau(A_i)\in\bar{B}'$ if $i>l$. Let $$\bar{B}'=\{\tau(A_{l+1}),\tau(A_{l+2}),\ldots,\tau(A_{s}),B_1,B_2,\ldots,B_t\}$$ and $$\bar{A}'-\theta(\bar{B}')=\{A_{m_1},\ldots,A_{m_r}\}.$$ For any $1\leq i\leq r$, we have a sequence $$A_{m_i},\tau(A_{m_i}),\theta\tau(A_{m_i}),\tau\theta\tau(A_{m_i}),\ldots,X,$$ where $X=\tau(A_p)$ with $p\leq l$ or $X=\theta(B_q)\in\bar{A}-\bar{A}'$.
In the first case, let
$$\bar{P}_i=\{(A_{m_i},\tau(A_{m_i})),(\theta\tau(A_{m_i}),\tau\theta\tau(A_{m_i})),\ldots,(A_p,\tau(A_p))\}.$$
In the second case, let
$$\bar{P}_i=\{(A_{m_i},\tau(A_{m_i})),(\theta\tau(A_{m_i}),\tau\theta\tau(A_{m_i})),\ldots,(A_q,B_q)\}.$$
Let $$\bar{P}_0=\{(\theta(X),X)|X\in\bar{B}'\textrm{ and $(-,X)\not\in\bar{P}_i$ for all $i$}\}.$$

Let $$\bar{P}=\bar{P}_0\cup\bar{P}_1\cup\cdots\cup\bar{P}_r.$$
Note that $\bar{P}$ is a partial matching between $\bar{A}$ and $\bar{B}$. Hence, we have $c(\bar{P})\geq\epsilon$ by definition.

Since $||a-b||_{\infty}\geq||\bar{a}-\bar{b}||_{\infty}$ for any $(a,b)\in P$, we have
\begin{equation}\label{eq_3.1}
\sup_{(a,b)\in{P}}
||a-b||_{\infty}\geq\sup_{(\bar{a},\bar{b})\in\bar{P}}||\bar{a}-\bar{b}||_{\infty}.
\end{equation}
Meanwhile, the point $\bar{s}$ is unmatched in $\bar{A}\cup\bar{B}$ implies that $s'$ is unmatched in $A\cup B$ for any $s'\in\bar{s}$.
Let $S$ and $\bar{S}$ be the sets of unmatched elements in $A\cup B$ and $\bar{A}\cup\bar{B}$ respectively.
Hence, we have
\begin{equation}\label{eq_3.2}
\sup_{s=(s_x,s_y)\in S} {\frac{s_x-s_y}{2}}\geq\sup_{\bar{s}=\overline{(s_x,s_y)}\in \bar{S}} {\frac{s_x-s_y}{2}}.
\end{equation}

Formulas \ref{eq_3.1} and \ref{eq_3.2} implies that $c(P)\geq c(\bar{P})$ and we get (\ref{eq_3.3}).

Then, we shall prove 
\begin{equation}\label{eq_3.4}
    d_{b,\mathbb{R}^2}(A,B)\leq d_{b,\mathbb{R}^2/{\sim}}(\bar{A},\bar{B}).
\end{equation}

Assume that $d_{b,\mathbb{R}^2}(A,B)=\epsilon$. Let $\bar{P}$ be a partial matching between $\bar{A}$ and $\bar{B}$. 
For any  $(\bar{a},\bar{b})\in\bar{P}$ and any $a'\in\bar{a}$, there exists $b'\in\bar{b}$ such that $||a'-b'||_{\infty}=||\bar{a}-\bar{b}||_{\infty}$. Let $P$ be the set of all $(a',b')$ considered above.
Since the sets $A$ and $B$ are $\sigma$-invariant, $\bar{a}$ is a subset of $A$, $\bar{b}$ is a subset of $B$ and $P$ is a partial matching between ${A}$ and ${B}$.
Hence, we have $c({P})\geq\epsilon$ by definition.

Since $||a'-b'||_{\infty}=||\bar{a}-\bar{b}||_{\infty}$ for any $(\bar{a},\bar{b})\in \bar{P}$, we have
\begin{equation}\label{eq_3.5}
\sup_{(a,b)\in{P}}
||a'-b'||_{\infty}=\sup_{(\bar{a},\bar{b})\in\bar{P}}||\bar{a}-\bar{b}||_{\infty}
\end{equation}
by constructions.
Meanwhile, the point $\bar{s}$ is unmatched in $\bar{A}\cup\bar{B}$ if and only if $s'$ is unmatched in $A\cup B$ for any $s'\in\bar{s}$.
Let $S$ and $\bar{S}$ be the sets of unmatched elements in $A\cup B$ and $\bar{A}\cup\bar{B}$ respectively.
Hence, we have
\begin{equation}\label{eq_3.6}
\sup_{s=(s_x,s_y)\in S} {\frac{s_x-s_y}{2}}=\sup_{\bar{s}=\overline{(s_x,s_y)}\in \bar{S}} {\frac{s_x-s_y}{2}}.
\end{equation}

Formulas \ref{eq_3.5} and \ref{eq_3.6} implies that $c(P)=c(\bar{P})$ and we get (\ref{eq_3.4}).

By Formulas \ref{eq_3.3} and \ref{eq_3.4}, we get the desired result.

\end{proof}

\begin{proposition}\label{Prop_geq}
For two pointwise finite-dimensional nilpotent representations  $\mathbf{V}$ and $\mathbf{W}$ of $\tilde{{A}}_{\mathbb{R}}$,
we have
\begin{equation*}
    d_{i,\tilde{{A}}_{\mathbb{R}}}(\mathbf{V},\mathbf{W})\geq d_{b,\mathbb{R}^2/{\sim}}(dgm(\mathbf{V}),dgm(\mathbf{W})).
\end{equation*} 
\end{proposition}
\begin{proof}
Since $\Psi$ is an equivalence between $\mathrm{Rep}^{pwf}_{k}(\tilde{{A}}_{\mathbb{R}})$ and $\mathrm{Rep}^{pwf}_{k}(\mathbf{Q})$, we have
\begin{equation}\label{formula_1}
    d_{i,\tilde{{A}}_{\mathbb{R}}}(\mathbf{V},\mathbf{W})=d_{i,\mathbf{Q}}(\Psi(\mathbf{V}),\Psi(\mathbf{W})).
\end{equation}
Note $\textrm{Hom}_{\mathbf{Q}}^{\varepsilon}(\Psi(\mathbf{V}),\Psi(\mathbf{W}))$ is a subset of $\textrm{Hom}_{{A}_{\mathbb{R}}}^{\varepsilon}(\Psi(\mathbf{V}),\Psi(\mathbf{W}))$. Hence,
\begin{equation*}
   d_{i,\mathbf{Q}}(\Psi(\mathbf{V}),\Psi(\mathbf{W}))\geq d_{i,{A}_{\mathbb{R}}}(\Psi(\mathbf{V}),\Psi(\mathbf{W})).
\end{equation*}    
By using Theorem \ref{thm_iso_A}, we have
\begin{equation*}
    d_{i,{A}_{\mathbb{R}}}( \Psi(\mathbf{V}),\Psi(\mathbf{W})) =
     d_{b,\mathbb{R}^2}(dgm(\Psi(\mathbf{V})),dgm(\Psi(\mathbf{W}))).
\end{equation*}
Since $dgm(\Psi(\mathbf{V}))$ and $dgm(\Psi(\mathbf{W}))$ are $\sigma$-invariant multisets such that $\overline{dgm(\Psi(\mathbf{V}))}$ and $\overline{dgm(\Psi(\mathbf{W}))}$ are finite, Lemma \ref{lemma-inv} implies that \begin{equation}\label{formula_2}
   d_{b,\mathbb{R}^2}(dgm(\Psi(\mathbf{V})),\Psi(dgm(\mathbf{W})))= d_{b,\mathbb{R}^2/{\sim}}(dgm(\mathbf{V}),dgm(\mathbf{W})).
\end{equation}
Hence, we get the desired result.

\end{proof}

\begin{lemma}[\cite{2015Persistence}]
\label{1}
    Let $|p,q|$, $|r,s|$ be two intervals. Then,
    \begin{equation*}
        d_{i,{A}_{\mathbb{R}}}(T_{|p,q|},T_{|r,s|})\leq ||a-b||_{\infty},
    \end{equation*}
     where $a=(p,q)$, $b=(r,s)$ are the corresponding points in $\mathbb{R}^2$.
\end{lemma}

\begin{lemma}\label{lemma-leq-1}
    Let $|p,q|$, $|r,s|$ be two intervals.
    Then,
    \begin{equation*}
        d_{i,\mathbf{Q}}(\mathbb{T}_{|p,q|},\mathbb{T}_{|r,s|})\leq ||\bar{a}-\bar{b}||_{\infty}.
    \end{equation*}
    where $a=(p,q)$, $b=(r,s)$ are the corresponding points in $\mathbb{R}^2$ and $\bar{a}$, $\bar{b}$ are the equivalent classes in $\mathbb{R}^2/{\sim}$.
\end{lemma}
\begin{proof}
By definition of $||\bar{a}-\bar{b}||_{\infty}$, we can choose $a'=(p',q')\in\bar{a}$ and $b'=(r',s')\in\bar{b}$ such that \begin{equation}\label{eq1}||a'-b'||_{\infty}=||\bar{a}-\bar{b}||_{\infty}.\end{equation}

By Lemma \ref{1}, there is a $\varepsilon$-interleaving $(\alpha,\beta)$ of $T_{|p',q'|}$ and $T_{|r',s'|}$ such that \begin{equation}\label{eq2}\varepsilon\leq||a'-b'||_{\infty}.\end{equation}
Since $$\mathbb{T}_{|p,q|}=\bigoplus_{k\in \mathbb{Z}}(\sigma^{*})^{k}(T_{|p',q'|})$$
and
$$\mathbb{T}_{|r,s|}=\bigoplus_{k\in \mathbb{Z}}(\sigma^{*})^{k}(T_{|r',s'|}),$$
 the pair  $(\bigoplus_{k\in \mathbb{Z}}(\sigma^{*})^{k}(\alpha),\bigoplus_{k\in \mathbb{Z}}(\sigma^{*})^{k}(\beta))$ of maps
is a $\varepsilon$-interleaving of $\mathbb{T}_{|p,q|}$ and $\mathbb{T}_{|r,s|}$.
Hence, 
\begin{equation}\label{eq3}
    d_{i,\mathbf{Q}}(\mathbb{T}_{|p,q|},\mathbb{T}_{|r,s|})\leq\varepsilon.
\end{equation}

Then, Formulas (\ref{eq1}), (\ref{eq2}) and (\ref{eq3}) imply the desired result.
  
\end{proof}

Similarly to the proof of the converse stability part of Theorem \ref{thm_iso_A} in \cite{2015Persistence}, we have the following lemma.

\begin{lemma}\label{lemma-leq}
For two representations $\mathbb{V}$ and $\mathbb{W}$ of $\mathbf{Q}$,
we have
\begin{equation*}
    d_{i,\mathbf{Q}}(\mathbb{V},\mathbb{W})\leq d_{b,\mathbb{R}^2}(dgm(\mathbb{V}),dgm(\mathbb{W})).
\end{equation*} 
\end{lemma}

\begin{proof}
Let $\mathbb{V}$ and $\mathbb{W}$ be two representations of $\textbf{Q}$ and $d_{b,\mathbb{R}^2}(dgm(\mathbb{V}),dgm(\mathbb{W}))=\varepsilon$. By definitions, the representations $\mathbb{V}$ and $\mathbb{W}$ can be wrote as direct sums of indecomposable interval representations as follows
$$\mathbb{V}=\bigoplus_{i\in I_1} \mathbb{V}_i\oplus\bigoplus_{j \in J} \mathbb{V}_j,$$ $$ \mathbb{W}=\bigoplus_{i\in I_2} \mathbb{W}_i\oplus\bigoplus_{j \in J} \mathbb{W}_j,$$
and there is a partial matching between the direct summands of $\mathbb{V}$ and those of $\mathbb{W}$
such that
\begin{enumerate}

    \item $\mathbb{V}_j$ is matched with $\mathbb{W}_j$ for any $j\in J$,
    \item $\mathbb{V}_i$ is unmatched for any $i\in I_1$,
    \item $\mathbb{W}_i$ is unmatched for any $i\in I_2$.
\end{enumerate}
By Lemma \ref{lemma-leq-1}, we have $d_{i,\mathbf{Q}}(\mathbb{V}_j,\mathbb{W}_j)\leq\varepsilon$. 

If $\mathbb{V}=\mathbb{V}_1\bigoplus\mathbb{V}_2$ and $\mathbb{W}=\mathbb{W}_1\bigoplus\mathbb{W}_2$, then
    \begin{equation*}
        d_{i,\mathbf{Q}}(\mathbb{V}_1\bigoplus\mathbb{V}_2,\mathbb{W}_1\bigoplus\mathbb{W}_2)\leq max\{d_{i,\mathbf{Q}}(\mathbb{V}_1, \mathbb{W}_1), d_{i,\mathbf{Q}}(\mathbb{V}_2,\mathbb{W}_2)\}.
    \end{equation*}
Hence, we have $d_{i,\mathbf{Q}}(\mathbb{V},\mathbb{W})\leq\varepsilon$.

\end{proof}

\begin{proposition}\label{Prop_leq}
For two representations $\mathbf{V}$ and $\mathbf{W}$ of $\tilde{{A}}_{\mathbb{R}}$.
we have
\begin{equation*}
    d_{i,\tilde{{A}}_{\mathbb{R}}}(\mathbf{V},\mathbf{W})\leq d_{b,\mathbb{R}^2/{\sim}}(dgm(\mathbf{V}),dgm(\mathbf{W}))
\end{equation*} 
\end{proposition}
\begin{proof}
By Lemma \ref{lemma-leq}, we have
\begin{equation*}
   d_{i,\mathbf{Q}}(\Psi(\mathbf{V}),\Psi(\mathbf{W}))\leq d_{b,\mathbb{R}^2}(dgm(\Psi(\mathbf{V})),dgm(\Psi(\mathbf{W}))).
\end{equation*}    
Then, the desired result is implied by Formulas (\ref{formula_1}) and (\ref{formula_2}).

\end{proof}

Propositions \ref{Prop_geq} and \ref{Prop_leq} imply Theorem \ref{main-thm}.

\bibliographystyle{abbrv}
\bibliography{ref}

\begin{thebibliography}{10}

\bibitem{Appel_2020}
A.~Appel and F.~Sala.
\newblock Quantization of continuum {K}ac–{M}oody algebras.
\newblock {\em Pure and Applied Mathematics Quarterly}, 16(3):439--493, 2020.

\bibitem{Appel_2022}
A.~Appel, F.~Sala, and O.~Schiffmann.
\newblock Continuum {K}ac–{M}oody algebras.
\newblock {\em Moscow Mathematical Journal}, 22(2):177--224, 2022.

\bibitem{2014Bauer-Lesnick}
U.~Bauer and M.~Lesnick.
\newblock Induced matchings of barcodes and the algebraic stability of persistence.
\newblock In {\em Proceedings of the Thirtieth Annual Symposium on Computational Geometry}, SOCG'14, page 355–364, New York, NY, USA, 2014. Association for Computing Machinery.

\bibitem{2018Decomposition}
M.~Botnan and W.~Crawley-Boevey.
\newblock Decomposition of persistence modules.
\newblock {\em Proceedings of the American Mathematical Society}, 148(11):4581--4596, 2020.

\bibitem{Botnan2018Algebraic}
M.~Botnan and M.~Lesnick.
\newblock Algebraic stability of zigzag persistence modules.
\newblock {\em Algebraic $\&$ Geometric Topology}, 18(6):3133–3204, 2018.

\bibitem{2017Interval}
M.~B. Botnan.
\newblock Interval decomposition of infinite zigzag persistence modules.
\newblock {\em Proceedings of the American Mathematical Society}, 145(8):3571--3577, 2017.

\bibitem{2014Bubenik-Scott}
P.~Bubenik and J.~A. Scott.
\newblock Categorification of persistent homology.
\newblock {\em Discrete $\&$ Computational Geometry}, 51:600--627, 2014.

\bibitem{2012The}
F.~Chazal, V.~De~Silva, M.~Glisse, and S.~Oudot.
\newblock {\em The structure and stability of persistence modules}.
\newblock Springer, 2016.

\bibitem{2007Cohen-Steiner-Edelsbrunner-Harer}
D.~Cohen-Steiner, H.~Edelsbrunner, and J.~Harer.
\newblock Stability of persistence diagrams.
\newblock {\em Discrete $\&$ Computational Geometry}, 37:103--120, 2007.

\bibitem{Crawley2015Decomposition}
W.~Crawley-Boevey.
\newblock Decomposition of pointwise finite-dimensional persistence modules.
\newblock {\em Journal of Algebra and its Applications}, 14(05):1550066, 2015.

\bibitem{Gabriel1972Unzerlegbare}
P.~Gabriel.
\newblock Unzerlegbare darstellungen {I}.
\newblock {\em Manuscripta mathematica}, 6:71--103, 1972.

\bibitem{2024Gao-Zhao}
X.~Gao and M.~Zhao.
\newblock Maximal rigid representations of continuous quivers of type {A} with automorphism.
\newblock {\em arXiv:2410.07210}, 2024.

\bibitem{2020Decomposition}
E.~J. Hanson and J.~D. Rock.
\newblock Decomposition of pointwise finite-dimensional ${S}^1$ persistence modules.
\newblock {\em Journal of Algebra and Its Applications}, 23(03):2450054, 2024.

\bibitem{Igusa2022Continuous}
K.~Igusa, J.~D. Rock, and G.~Todorov.
\newblock Continuous quivers of type ${A}$ ({I}) foundations.
\newblock {\em Rendiconti del Circolo Matematico di Palermo Series 2}, 72(2):833--868, 2023.

\bibitem{Igusa2013}
K.~Igusa and G.~Todorov.
\newblock {\em Continuous Frobenius Categories}, pages 115--143.
\newblock Springer, Berlin, Heidelberg, 2013.

\bibitem{lesnick2015}
M.~Lesnick.
\newblock The theory of the interleaving distance on multidimensional persistence modules.
\newblock {\em Foundations of Computational Mathematics}, 15(3):613–650, 2015.

\bibitem{2015Persistence}
S.~Y. Oudot.
\newblock {\em Persistence theory: from quiver representations to data analysis}, volume 209.
\newblock American Mathematical Society, 2017.

\bibitem{rock2023continuous}
J.~D. Rock and S.~Zhu.
\newblock Continuous nakayama representations.
\newblock {\em Applied Categorical Structures}, 31(5):44, 2023.

\bibitem{Sala_2019}
F.~Sala and O.~Schiffmann.
\newblock The circle quantum group and the infinite root stack of a curve.
\newblock {\em Selecta Mathematica}, 25(77):1--86, 2019.

\bibitem{Sala_2021}
F.~Sala and O.~Schiffmann.
\newblock Fock space representation of the circle quantum group.
\newblock {\em International Mathematics Research Notices}, 2021(22):17025--17070, 2021.

\end{thebibliography}




\end{document}